\documentclass[12pt]{article}

\usepackage{wrapfig, rotating}

\usepackage[utf8]{inputenc} 
\usepackage{textcomp} 
\usepackage{graphicx,epsfig}  
\usepackage{epstopdf} 
\usepackage{flafter}  
\usepackage[numbers]{natbib} 

\usepackage{amsmath,amssymb,amsthm}
\usepackage{bm}  
\usepackage{enumerate}
\usepackage{subfigure}

\newcommand{\IRKA}{\textsf{\small{IRKA}}}

\newcommand{\ZIP}{\textsf{\small{ZIP}}}

\newcommand{\sgn}{\mathop{\mathrm{sgn}}}

\newcommand{\Hs }{\ensuremath{H(s)}}
\newcommand{\Hr}{\ensuremath{H_r(s)}}

\newcommand{\boldA}{\ensuremath{\boldsymbol A}}
\newcommand{\boldP}{\ensuremath{\boldsymbol P}}
\newcommand{\boldx}{\ensuremath{\boldsymbol x}}
\newcommand{\boldAr}{\ensuremath{\boldsymbol{A}_r}}
\newcommand{\boldb}{\ensuremath{\boldsymbol b}}
\newcommand{\boldbr}{\ensuremath{\boldsymbol {b}_r}}

\newcommand{\boldc}{\ensuremath{\boldsymbol c}}
\newcommand{\boldcr}{\ensuremath{\boldsymbol{c}_r}}

\newcommand{\boldWr}{\ensuremath{\boldsymbol{W}_r}}

\newcommand{\boldVr}{\ensuremath{\boldsymbol{V}_r}}

\newcommand{\complex}{\ensuremath{\mathbb{C}}}
\newcommand{\boldZr}{\ensuremath{\boldsymbol Z}_r}

\newfont{\Bb}{msbm10 scaled\magstep0}
\def\IR{\mbox {\Bb R}}

\newcommand{\Hilbspace}{\ensuremath{\mathcal{H}_2 \text{ }}}

\newcommand{\ud}{\,\mathrm{d}}

\newcommand{\Qr}{\ensuremath{\boldsymbol{Q}_r}}

\newcommand{\boldI}{\ensuremath{\boldsymbol{I}}}

\newcommand{\costfun}{\ensuremath{\boldsymbol{\mathcal{J}}}}
\newcommand{\poles}{\ensuremath{\boldsymbol{s}}}

\newcommand{\SSS}{\textsf{\small SSS}}

\newcommand{\Htwo}{\ensuremath{\mathcal{H}_{2}}}

\newtheorem{thm}{Theorem}[section]
\newtheorem{lemma}{Lemma}[section]

\newtheorem{prop}{Proposition}[section]

\theoremstyle{definition}
\newtheorem{dfn}{Definition}
\theoremstyle{plain}

\title{Convergence of the Iterative Rational Krylov Algorithm}
\author{Garret Flagg, Christopher Beattie, Serkan Gugercin \\ 
{\small Department of Mathematics, Virginia Tech.} \\
{\small Blacksburg, VA, 24061-0123} \\
{\small \tt e-mail: \{flagg,beattie,gugercin\}@math.vt.edu}
}

\begin{document}
\maketitle
\begin{abstract}

Iterative Rational Krylov Algorithm (\IRKA) of \cite{H2} is an interpolatory model reduction approach to optimal $\Htwo$ approximation problem. Even though the method has been illustrated to show rapid convergence in various examples, a proof of convergence has not been provided yet. In this note, we show that in the case of state-space symmetric systems, \IRKA~is a  locally convergent fixed point iteration to a local minimum of the underlying $\Htwo$ approximation problem.

\end{abstract}

\section{Introduction}


Consider a single-input-single-output (SISO) linear dynamical system in state-space form:
\begin{eqnarray}
\label{ltisystemintro}
\dot{\boldx}(t) = \boldA \,\boldx(t) + \boldb\,u(t),\qquad
y(t) = \boldc^T \boldx(t),
\end{eqnarray}
where $\boldA \in \IR^{n \times n}$, and $\boldb,\boldc \in \IR^{n}$. In (\ref{ltisystemintro}),
$\boldx(t)\in \IR^n$, $u(t)\in \IR$, $y(t)\in \IR$, are, respectively, the 
\emph{states}, \emph{input}, and \emph{output} of the dynamical system.
The transfer function of the underlying system is 
$H(s) = \boldc^T(s\boldI-\boldA)^{-1}\boldb$.  
$H(s)$  will be used to denote both the system and its transfer function.

Dynamical systems of the form (\ref{ltisystemintro}) with large state-space dimension $n$ appear in many applications; see, e.g.,  \cite{antB} and \cite{KorR05}. Simulations in such large-scale
settings make enormous demands on computational
resources.  The goal of model reduction is  to  construct a surrogate system
\begin{eqnarray} \label{redsysintro}
\dot{\boldx}_r (t) = \boldA_r\, \boldx_r (t) + \boldb_r u(t),\qquad
y_r(t)  =  \boldc_r^T \boldx_r (t),
\end{eqnarray}
of much smaller dimension $r \ll n$, with  $\boldA_r \in \IR^{r \times r}$ and
$\boldb_r,\,\boldc_r \in \IR^{r}$ such that $y_r(t)$ approximates $y(t)$ well in a certain norm. 
Similar to $H(s)$, the transfer function $H_r(s)$ of the reduced-model (\ref{redsysintro}) is given by 
$H_r(s) = \boldc_r^T(s\boldI_r-\boldA_r)^{-1}\boldb_r$.   We consider reduced-order
models, $H_r(s)$, that are obtained via  projection. 
That is, we choose full rank matrices 
 $\boldVr,\boldWr \in \IR^{n \times r}$ 
such that $\boldWr^T\boldVr$ is invertible and define the reduced-order state-space realization with (\ref{redsysintro}) and 
{\small
\begin{equation}  \label{red_projection}
\boldAr = (\boldsymbol{W}_r^T\boldsymbol{V}_r)^{-1}\boldWr^T \boldA \boldVr,~\boldbr = (\boldsymbol{W}_r^T\boldsymbol{V}_r)^{-1}\boldWr^T\boldb,~
\boldcr = \boldVr^T\boldc.
\end{equation}
}
Within this ``projection framework," selection of $\boldWr$ and $\boldVr$  completely determines the reduced system -- indeed, it is sufficient to specify only the \emph{ranges} of $\boldWr$ and $\boldVr$ in order to determine $H_r(s)$.   Of particular utility for us is a result by Grimme \cite{Grim}, that gives conditions on $\boldWr$ and $\boldVr$ so that the associated reduced-order system, 
$H_r(s)$, is a \emph{rational Hermite interpolant} to the original system, $H(s)$.  
%
%
\begin{thm}[Grimme \cite{Grim}]  \label{thm:rationalkrylov}
Given $\Hs=\boldc^T(s\boldI-\boldA)^{-1}\boldb$, and $r$ distinct points $s_1, \dots, s_{r} \in \complex$, let 
\begin{equation} \label{eqn:VrWr}
\boldsymbol{V}_r=\lbrack (s_1\boldsymbol{I} -\boldsymbol{A})^{-1}\boldsymbol{b} \dots  (s_r\boldsymbol{I} -\boldsymbol{A})^{-1}\boldsymbol{b}\rbrack
\qquad
\boldsymbol{W}_r^T=
\begin{bmatrix}
\boldsymbol{c}^T(s_{1}\boldsymbol{I}-\boldsymbol{A})^{-1}\\
\vdots \\
\boldsymbol{c}^T(s_{r}\boldsymbol{I}-\boldsymbol{A})^{-1}
\end{bmatrix}.
\end{equation}

Define the reduced-order model $H_r(s) = \boldcr^T(s \boldI_r - \boldAr)^{-1}\boldbr$ as in (\ref{red_projection}) 
Then $H_r$ is a rational Hermite interpolant to $H$ at $s_1, \dots, s_{r}$:
\begin{equation} \label{eqn:hermite}
H(s_i) = H_r(s_i) \qquad {\rm and} \qquad H'(s_i) = H'_r(s_i)~~~{\rm for}~~~i=1,\ldots,r. 
\end{equation}
\end{thm}
Rational interpolation within this ``projection framework" was first proposed by Skelton \emph{et al}. \cite{Skeltonlate},\cite{Skelt1},\cite{Skelt2}.  
Later in  \cite{Grim}, Grimme established the connection with the rational Krylov method of Ruhe \cite{ruhe1998rational}.

Significantly, Theorem \ref{thm:rationalkrylov} gives an explicit method for computing a reduced-order system that is a Hermite interpolant of the orginal system for nearly \emph{any} set of distinct points, $\{s_1, \dots, s_{r} \}$, yet it is not apparent how one should choose these interpolation points in order to assure a high-fidelity reduced-order model in the end.  Indeed, the lack of such a strategy had been a major drawback 
for interpolatory model reduction until recently, when an effective strategy for selecting interpolation points was proposed in \cite{H2} yielding reduced-order models that solve
\begin{equation} \label{optimalHtwo}
\| H - H_r \|_{\Htwo} = 
\min \limits_{dim(\hat{H}_r) = r}  \left\| H-\hat{H}_r \right\|_{\Htwo}.
\end{equation}
where the \Hilbspace~system norm is defined in the usual way: 
\begin{equation}
\left \| H \right \|_{\Htwo} = \left(\frac{1}{2\pi} \int_{-\infty}^\infty \mid H(\jmath \omega) \mid^2 d \omega \right)^{1/2}.
\end{equation}

The optimization problem (\ref{optimalHtwo}) has been studied extensively, 
see, for example,
\cite{meieriii1967approximation,wilson1970optimum,H2,spanos1992anewalgorithm, fulcheri1998mrh,vandooren2008hom,gugercin2005irk,beattie2007kbm,beattie2009trm,zigic1993contragredient} and references therein. (\ref{optimalHtwo}) is a nonconvex optimization problem and finding global minimizers will be infeasible, typically.  Hence, the usual interpretation of (\ref{optimalHtwo}) involves finding \emph{local} minimizers and a common approach to accomplish this is to construct reduced-order models satisfying first-order necessary optimality conditions.  This may be posed either in terms of solutions to Lyapunov equations (e.g., \cite{wilson1970optimum,spanos1992anewalgorithm, zigic1993contragredient}) 
or in terms of interpolation (e.g., \cite{wilson1970optimum,H2,vandooren2008hom, beattie2009trm}): 
\begin{thm}(\cite{meieriii1967approximation,H2}) \label{thm:h2optimal}
Given \Hs, let $H_r(s)$ be a solution to (\ref{optimalHtwo}) with simple poles $\hat{\lambda}_1,\dots,\hat{\lambda}_r$. Then 
\begin{equation} \label{eqn:h2cond}
H(-\hat{\lambda}_i)=H_r(-\hat{\lambda}_i)~~~{\rm and}~~~
H'(-\hat{\lambda}_i)=H_r'(-\hat{\lambda}_i)~~~{\rm for}~~~
 i=1,\dots,r.
\end{equation}
That is, any $\mathcal{H}_2$-optimal reduced order model of order $r$ with simple poles will be a Hermite interpolant to $H(s)$
at the reflected image of the reduced poles through the origin. 
\end{thm}

Although this result might appear to reduce the problem of $\mathcal{H}_2$-optimal model approximation to a straightforward application of Theorem \ref{thm:rationalkrylov} to calculate a Hermite interpolant on the set of reflected poles, $\{-\hat{\lambda}_1,\dots,-\hat{\lambda}_r\}$, these pole locations will not be known \emph{a priori}.   Nonetheless, these pole locations can be determined efficiently with the Iterative Rational Krylov Algorithm ({\IRKA}) of Gugercin \emph{et al.} \cite{H2}.  Starting from an arbitrary initial selection of interpolation  points, \IRKA~iteratively corrects the interpolation points until (\ref{eqn:h2cond}) is satisfied. A brief sketch of \IRKA~is given below.
\begin{figure}[ht]
\begin{center}
    \framebox[6.25in][t]{
    \begin{minipage}[c]{5.75in}
    {
\textbf{\textsf{\small{Algorithm}}~\IRKA}.  \emph{Iterative Rational Krylov Algorithm \cite{H2}} 

\quad Given a full-order $H(s)$, a reduction order $r$,  and convergence tolerance $\mathsf{tol}$.
\begin{enumerate}
\item Make an initial selection of $r$ distinct interpolation points, $\{s_i\}_1^r$, that is closed under complex conjugation.
\item Construct \boldVr\ and \boldWr\ as in (\ref{eqn:VrWr}).
\item while (relative change in $\{s_i\} > tol$)
\begin{enumerate}[a.)]
\item $\boldAr= (\boldsymbol{W}_r^T\boldsymbol{V}_r)^{-1}\boldWr^T \boldA\boldVr$
\item Solve $r\times r$ eigenvalue problem $\boldAr\mathbf{u}=\lambda\mathbf{u}$ and assign $s_i \leftarrow -\lambda_i(\boldAr)$ for $i=1,\dots ,r$.
\item Update \boldVr\ and \boldWr\  with new $s_i$'s using (\ref{eqn:VrWr}).
\end{enumerate}
\item $\boldAr=(\boldsymbol{W}_r^T\boldsymbol{V}_r)^{-1}\boldWr^T$\boldA\boldVr, ~~~$\boldbr=(\boldsymbol{W}_r^T\boldsymbol{V}_r)^{-1}\boldWr^T\boldb$, ~and~ $\boldcr=\boldVr^T\boldc$.
\end{enumerate}
 }
    \end{minipage}
    }
    \end{center}
  \end{figure}

 \IRKA\ has been remarkably successful in producing high fidelity reduced-order approximations and has
been successfully applied to finding $\Htwo$-optimal reduced models for systems of high order, $n>160,000$, see \cite{KRXC08}. 
For details on \IRKA, see \cite{H2}.   

Notwithstanding typically observed rapid convergence of the \IRKA~iteration to interpolation points that generally yield high quality reduced models, no convergence theory for \IRKA~has yet been established.  Evidently from the description above, \IRKA~may be viewed as a fixed point iteration with fixed points coinciding with the stationary points of the $\Htwo$ minimization problem.  Saddle points and local maxima  of the $\Htwo$ minimization problem are known to be repellent \cite{krajewski}. However, despite effective performance in practice,  it has not yet been established that local minima are attractive fixed points.  

 In this paper, we give a proof of this for the special case of state-space-symmetric systems and establish the convergence of \IRKA~for this class of systems.   





\section{State-Space-Symmetric Systems}
 \begin{dfn}\label{dfn:SSS}
\Hs=$\boldc^T(s\boldI-\boldA)^{-1}\boldb$ is \emph{state-space-symmetric} (\SSS)  if \boldA=$\boldA^T$ and $\boldc=\boldb$.
\end{dfn}
\SSS~systems appear in many important applications such as in the analysis of RC circuits and in inverse problems involving 3D Maxwell's equations \cite{druskin2009solution}. 

A closely related class of systems is the class of zero-interlacing-pole (\ZIP) systems.
\begin{dfn}\label{dfn:ZIP}
A system
${\displaystyle \Hs=K\frac{\prod\limits_{i=1}^{n-1}(s-z_i)}{\prod\limits_{j=1}^{n}(s-\lambda_j)} }$
is a \emph{strictly proper \ZIP~system} provided that $$0>\lambda_1 >z_1 >\lambda_2 > z_2 > \lambda_3 > \dots > z_{n-1} > \lambda_n.$$  
\end{dfn}

The following relation serves to characterize \ZIP~systems.
\begin{prop} \cite{Zipchar} \label{prop:zipchar}
\Hs\ is a strictly proper \ZIP~system if and only if \Hs~can be written as
${\displaystyle \Hs=\sum_{i=1}^{n}\frac{b_i}{s-\lambda_i}}$ with $\lambda_i < 0$,  $b_i>0$, and $ \lambda_i\ne \lambda_j$  
 for all $ i\ne j.$
\end{prop}
 
The next result clarifies the relationship between \SSS~and \ZIP~systems.
 
 \begin{lemma}\label{lemma:reducessslemma} \cite{zippreserve}
Let  \Hs\  be  \SSS. Then \Hs\ is minimal if and only if the poles of \Hs\ are distinct.  Moreover, every \SSS~system  has a \SSS~minimal realization with distinct poles, and is therefore a strictly proper \ZIP~system.
 \end{lemma}

It can easily be verified from the implementation of \IRKA~given above, that for \SSS~systems, the relationship \boldVr=\boldWr\ is maintained throughout the iteration, and the final reduced-order model at Step $4$ of \IRKA~can be obtained by 
\begin{equation}\label{symmetry_preserve}
\begin{array}{cc}
\boldAr=\Qr^T\boldA\Qr&\boldbr=\boldcr=\Qr^T\boldb,
\end{array}
\end{equation}
where \Qr\ is an orthonormal basis for \boldVr;  the reduced system resulting from \IRKA~is also \SSS.  

\section{The Main Result}  \label{sec:proof_thm:localattractor}
\begin{thm}\label{thm:localattractor}
Let \IRKA~be applied to a minimal \SSS~system $H(s)$.  Then every fixed point of \IRKA~which is a local minimizer is locally attractive. In other words, \IRKA~is a locally convergent fixed point iteration to a local minimizer of the $\Htwo$
optimization problem.
\end{thm}

To proceed with the proof of Theorem \ref{thm:localattractor}, we need four intermediate lemmas.
 The first lemma provides insight into the structure of the zeros of the error system resulting from
 reducing a \SSS~system.
\begin{lemma}\label{lemma:zero_structure}
Let \Hs\ be a \SSS~system of order $n$.  If $H_r(s)$ is a ZIP system  that interpolates \Hs\ at $2r$ points
$s_1,s_2,\ldots,s_{2r}$, not necessarily distinct, in $(0, \infty),$ then all the remaining zeros of the error system lie in $(-\infty,0)$.
\end{lemma}
\begin{proof}  
By Lemma \ref{lemma:reducessslemma}, we may assume that \Hs\ is a strictly proper \ZIP~systems.  Since \Hs\ is a strictly proper \ZIP~system, its poles are simple and all its residues are positive.  Let
 $\lambda_i <0, \phi_i>0,$ for $i=1,\dots,n$ be the poles and residues of $H(s)$, respectively.  Now let 
$$R(s)=\prod \limits_{i=1}^{2r}(s-s_i),~~P(s)=\prod\limits_{i=1}^{n-r-1}(s+z_i),~~Q(s)=\prod\limits_{i=1}^{n}(s-\lambda_i),~~\tilde{Q}(s)=\prod \limits_{i=1}^{r}(s-\tilde{\lambda}_i),$$ 
where $\tilde{\lambda}_i$, $s_i$, and $z_i$ are, respectively, the poles of \Hr, the interpolation points, and the remaining zeros of the error system. Then for some constant $K$, $\Hs-H_r(s)=K\frac{P(s)R(s)}{Q(s)\tilde{Q}(s)}$. First suppose that $\{\lambda_i\}_{i=1}^n \cap \{\tilde{\lambda}_k\}_{k=1}^r=\emptyset$.  Then for each $\lambda_j$, $j=1,\dots,n$, 
 \begin{equation}
 \text{Res}(\Hs-H_r(s);\lambda_j)=K\frac{P(\lambda_j)R(\lambda_j)}{\prod\limits_{\substack{i=1\\ \lambda_i\ne\lambda_j}}^{n}(\lambda_j-\lambda_i)\tilde{Q}(\lambda_j)}=\phi_i>0.
 \end{equation}
Thus, $\sgn(KP(\lambda_j))=(-1)^{j-1}\sgn(\tilde{Q}(\lambda_j))$ where $\sgn(\alpha)$ denotes the sign of
$\alpha$. Now if $(-1)^{j-1}\sgn(\tilde{Q}(\lambda_j))=(-1)^{j}(\sgn(\tilde{Q}(\lambda_{j+1}))$, then $-\sgn(\tilde{Q}(\lambda_j))= \sgn(\tilde{Q}(\lambda_{j+1}))$, so $\tilde{Q}(s)$ must change sign on the interval $[\lambda_{j+1}, \lambda_j]$.  Since $\tilde{Q}(s)$ is a polynomial of degree $r$, and $r<n$, $\tilde{Q}(s)$ can switch signs at most $r$ times, else $\tilde{Q}(s)\equiv 0$.  But this means there are at least $n-r-1$ intervals $[\lambda_{j_k+1}, \lambda_{j_k}]$, for $k=1, \dots, n-r-1$, for which  $\sgn(\tilde{Q}(\lambda_{j_k}))= \sgn(\tilde{Q}(\lambda_{j_k+1}))$, and therefore $\sgn(KP(\lambda_{j_k}))=-\sgn(KP(\lambda_{j_k+1}))$. So $KP(s)$ must change sign over at least $n-r-1$ intervals, and therefore has at least $n-r-1$ zeros on $[\lambda_n, \lambda_1]$.  Again, since the error is not identically zero when $r<n$, and the degree of $KP(s)$ is $n-r-1$, this implies that all the zeros of $KP(s)$ lie in $(-\infty, 0)$.
 
 Suppose with some $p \le r$, $\lambda_{i_j}=\tilde{\lambda}_{k_j}$ for $j=1, \dots, p$.   Observe from partial fraction expansions of $\Hs$ and $H_r(s)$ that the error can be written as a rational function of degree $n+r-p-1$ over degree $n+r-p$ with distinct poles. $n$ of these poles belong to $H(s)$ and the remaining $r-p$ come from the poles of $H_r(s)$ that are distinct from the poles of $H(s)$.  Now let $$R(s)=\prod \limits_{i=1}^{2r}(s-s_i),~~P(s)=\prod\limits_{i=1}^{n-r-p-1}(s+z_i),~~Q(s)=\prod\limits_{i=1}^{n}(s-\lambda_i),~~\tilde{Q}(s)=\prod \limits_{l=1}^{r-p}(s-\tilde{\lambda}_{k_l}),$$ where $\{\tilde{\lambda}_{k_l}\}_{l=1}^{r-p}=\{\tilde{\lambda}_k\}_{k=1}^r\setminus \{\lambda_i\}_{i=1}^n$.  
Hence, $\Hs-H_r(s)=K\frac{P(s)R(s)}{Q(s)\tilde{Q}(s)}$.  Observe that there are at most $2p$ subintervals of the form $[\lambda_{i^*}, \lambda_{i^*+1}]$ or $[\lambda_{i^*-1}, \lambda_{i^*}]$, where $\lambda_{i^*} \in \{\lambda_i\}_{i=1}^n \cap \{\tilde{\lambda}_k\}_{k=1}^r$.  It follows that there are at least $n-2p-1$ subintervals of the form $[\lambda_i, \lambda_{i+1}]$, where $\lambda_i, \lambda_{i+1} \not \in \{\lambda_i\}_{i=1}^n \cap \{\tilde{\lambda}_k\}_{k=1}^r$.
On each such subinterval for which this is the case, we have  

\begin{equation}
 \text{Res}(\Hs-H_r(s);\lambda_i)=K\frac{P(\lambda_i)R(\lambda_i)}{\prod\limits_{\substack{j=1\\ \lambda_j\ne\lambda_i}}^{n}(\lambda_i-\lambda_j)\tilde{Q}(\lambda_i)}=\phi_i>0.
 \end{equation}

So $\sgn(KP(\lambda_i))=(-1)^{i-1}\sgn(\tilde{Q}(\lambda_i))$.  By the same argument as above where the poles of $H(s)$ and $H_r(s)$ are distinct, either $\tilde{Q}(s)$ or $P(s)$ has a zero on the interval $[\lambda_i, \lambda_{i+1}]$.  Since $\tilde{Q}(s)$ has at most $r-p$ zeros, this means that there are at least $n-2p-1-(r-p)=n-p-r-1$ subintervals between poles of $H(s)$ where $P(s)$ has zeros. Hence, the lemma is proved.
\end{proof}

\begin{lemma}\label{lemma:positiveerror}
Let $H(s)=\boldb^T(s\boldI-\boldA)^{-1}\boldb$ be \SSS, and $H_r(s)=\boldbr^T(s\boldI_r-\boldAr)^{-1}\boldbr$ be any reduced order model of \Hs\ constructed by a compression of \Hs, i.e., \boldAr=$\Qr^T\boldAr\Qr$, \boldbr$=\Qr^T\boldb$.  Then for any $s \ge 0$, $\Hs-H_r(s) \ge 0$.  
\end{lemma}

\begin{proof} 
Pick any $s \ge 0$.  Then $(s\boldI_n-\boldA)$ is symmetric, positive definite and has a Cholesky decomposition, $(s\boldI_n-\boldA)= \boldsymbol{LL}^T$. 
Define $\boldZr=\boldsymbol{L}^T\Qr$.   Then
\begin{align*}
H(s)-H_r(s) =&\ \boldb^T\left[ (s\boldI_n-\boldA)^{-1} -\Qr\left(\Qr^T(s\boldI_n-\boldA)\Qr\right)^{-1}\Qr^T \right]\boldb \\
   =&\ (\boldsymbol{L}^{-1}\boldb)^T\left[ \boldI -\boldZr \left(\boldZr^T\boldZr\right)^{-1}\boldZr^T \right](\boldsymbol{L}^{-1}\boldb).
\end{align*}
 Note the last bracketed expression is an orthogonal projector onto $\mathsf{Ran}(\boldZr)^\perp$, hence is positive semidefinite and the conclusion follows. 
\end{proof}

Our convergence analysis  of \IRKA~will use its formulation as a fixed-point iteration.  The analysis 
will build on the framework of  \cite{krajewski}.  Let
\begin{equation}
\Hs=\sum \limits_{i=1}^n\frac{\phi_i}{s-\lambda_i} {\rm~~~and~~~}  H_r(s)=\sum\limits_{j=1}^{r}\frac{\tilde{\phi}_j}{s-\tilde{\lambda}_j}
\end{equation}
be the partial fraction decompositions of \Hs, and $H_r(s)$, respectively.  
Given a set of $r$ interpolation points $\{s_i\}_{i=1}^r$, identify the set with a vector $\poles=[s_1, \dots, s_r]^T$. Construct an interpolatory reduced order model \Hr\ from \poles\, as in  Theorem \ref{thm:rationalkrylov} and identify $\{\tilde{\lambda}_i\}_{i=1}^r$ with a vector $\tilde{\boldsymbol{\lambda}}=[\tilde{\lambda}_1, \dots, \tilde{\lambda}_r]^T$.  Then define the function $\lambda:\complex^{r} \rightarrow \complex^r$ by $\lambda(\poles)=-\tilde{\boldsymbol{\lambda}}$.  Aside from ordering issues, this function is well defined, and the \IRKA~iteration converges when $\lambda(\poles)=\poles$.  Thus convergence of \IRKA~is equivalent to convergence of a fixed point iteration on the function $\lambda(\poles)$. 
Similar to $\poles$ and $\tilde{\boldsymbol{\lambda}}$, let  $\tilde{\boldsymbol{\phi}} = [\tilde{\phi}_1, \dots, \tilde{\phi}_r]^T$.
 Having identified \Hr\ with its poles and residues, the optimal $\Htwo$ model reduction problem may be formulated in terms of minimizing the cost function  $\costfun(\tilde{\boldsymbol{\phi}}, \lambda(\poles)) = \|H-H_r \|_{\Hilbspace}^2$, where 
\begin{align} \label{h2error}
\costfun(\tilde{\boldsymbol{\phi}}, \lambda(\poles))&=\sum \limits^{n}_{i=1} \phi_i(H(\lambda_i)-H_r(\lambda_i)) + \sum \limits^{r}_{j=1} \tilde{\phi}_j(H(\tilde{\lambda}_j)-H_r(\tilde{\lambda}_j))
\end{align}
See \cite{H2} for a derivation of (\ref{h2error}).
Define the matrices 
$\boldsymbol{S}_{11},\boldsymbol{S}_{12},\boldsymbol{S}_{22} \in \IR^{r \times r}$ as 
$$
\left[\boldsymbol{S}_{11}\right]_{i,j} =  -(\tilde{\lambda}_i+\tilde{\lambda}_j)^{-1},~~\left[\boldsymbol{S}_{12}\right]_{i,j} = -(\tilde{\lambda}_i+\tilde{\lambda}_j)^{-2}~~{\rm and}~~
\left[\boldsymbol{S}_{22}\right]_{i,j} =  -2(\tilde{\lambda}_i +\tilde{\lambda}_j)^{-3}
$$
for $i,j=1,\ldots,r$. Also, define $\boldsymbol{R},\boldsymbol{E} \in \IR^{r\times r}$:
 $$\boldsymbol{R}={\rm diag}(\{\tilde{\phi}_1, \dots, \tilde{\phi}_r\}),~~~{\rm and}~~~\boldsymbol{E}={\rm diag}(\{H''(-\tilde{\lambda}_1)-H_r^{\prime\prime}(-\tilde{\lambda}_1), \dots, H''(-\tilde{\lambda}_r)-H_r^{\prime\prime}(-\tilde{\lambda}_r)\}.$$  


\begin{lemma}\label{lemma:E_pos_def}
Let \Hs\ be \SSS\ and let $H_r(s)$ be an \IRKA~interpolant.  Then $\boldsymbol{E}$ is positive definite
at any fixed point of $\lambda(\poles)$.
\end{lemma}

\begin{proof}
By Lemma \ref{lemma:positiveerror}, $\Hs-H_r(s) \ge 0$ for all $s\in [0, \infty)$.  Thus the points $H(-\tilde{\lambda}_i)-H_r(-\tilde{\lambda}_i)$ are local minima of $\Hs-H_r(s)$ on $[0, \infty)$ for $i=1,\dots,r$.  It then follows that $H''(-\tilde{\lambda}_i)-H_r^{\prime\prime}(-\tilde{\lambda}_i)\ge 0$.  But by Lemma \ref{lemma:zero_structure}, $\Hs-H_r(s)$ has exactly $2r$ zeros in $\complex_{+}$, so $H''(-\tilde{\lambda}_i)-H_r^{\prime\prime}(-\tilde{\lambda}_i) > 0$ for $i=1, \dots, r$.
\end{proof}


\begin{lemma}\label{lemma:Spos}
The matrix
$
\tilde{\boldsymbol{S}}=
\begin{bmatrix}
\boldsymbol{S}_{11} &\boldsymbol{S}_{12}\\
\boldsymbol{S}_{12}& \boldsymbol{S}_{22}
\end{bmatrix} 
$
is positive definite.
\end{lemma}

\begin{proof}
  We will show that for any non-zero vector $\boldsymbol{z}=[z_1, z_2,\dots, z_{2r}]^T \in \IR^{2r}$
$$
\boldsymbol{z}^T\boldsymbol{S}\boldsymbol{z} = \int_{0}^{\infty}\bigg[\sum\limits_{i=1}^r z_ie^{\tilde{\lambda}_i t}-t\Big(\sum \limits_{i=1}^rz_{r+i}e^{\tilde{\lambda}_i t}\Big) \bigg]^2 \ud t >0.
$$

Define $\boldsymbol{z}_r=[z_1, z_2,\dots, z_{r}]^T \in \IR^{r}$  and $\boldsymbol{z}_{2r}=[z_{r+1}, z_{r+2},\dots, z_{2r}]^T \in \IR^{r}$. Then 

\begin{equation}
\boldsymbol{z}^T\boldsymbol{S}\boldsymbol{z}= \boldsymbol{z}_r^T\boldsymbol{S}_{11}\boldsymbol{z}_r +2\boldsymbol{z}_r^T\boldsymbol{S}_{12}\boldsymbol{z}_{2r}+\boldsymbol{z}_{2r}^T\boldsymbol{S}_{22}\boldsymbol{z}_{2r} \label{Seqn}
\end{equation}

Let $\boldsymbol{\Lambda}= \text{diag}(\tilde{\lambda}_1, \dots, \tilde{\lambda}_r)$ and $\boldsymbol{u}$ be a vector of $r$ ones. Note that $\boldsymbol{S}_{11}$ solves the Lyapunov equation $\boldsymbol{\Lambda} \boldsymbol{S}_{11} + \boldsymbol{S}_{11} \boldsymbol{\Lambda} + 
 \boldsymbol{uu}^T = \boldsymbol{0}$.  Thus, 
\begin{align}
\boldsymbol{z}_r^T\boldsymbol{S}_{11}\boldsymbol{z}_r&=\int_0^{\infty}\boldsymbol{z}_r^Te^{\boldsymbol{\Lambda} t}\boldsymbol{uu}^Te^{\boldsymbol{\Lambda} t} \boldsymbol{z}_r \ud t
= \int_{0}^{\infty} \Big(\sum \limits_{i=1}^r z_i e^{\tilde{\lambda}_i t} \Big)^2 \ud t \label{S11eqn}
\end{align}

Similarly, $\boldsymbol{S}_{12}$ solves $\boldsymbol{\Lambda} \boldsymbol{S}_{12} + \boldsymbol{S}_{12} \boldsymbol{\Lambda} - \boldsymbol{S}_{11} = \boldsymbol{0}$. An application of integration by parts gives:
\begin{align}
\int_{0}^{\infty}t \Big(\sum \limits_{i=1}^{r} z_i e^{\tilde{\lambda}_i t}\Big)\Big(\sum \limits_{i=1}^{r} z_{r+i} e^{\tilde{\lambda}_i t}\Big)\ud t&=\int_{0}^{\infty}t(\boldsymbol{z}_r^T(e^{\boldsymbol{\Lambda} t}\boldsymbol{uu}^Te^{\boldsymbol{\Lambda} t})\boldsymbol{z}_{2r})\ud t \nonumber \\
&=\boldsymbol{z}_r^T\Big[-t (e^{\boldsymbol{\Lambda} t}\boldsymbol{S}_{11}e^{\boldsymbol{\Lambda} t})\Big]_{0}^{\infty}\boldsymbol{z}_{2r}+\boldsymbol{z}_r^T\Big(\int_{0}^{\infty} e^{\boldsymbol{\Lambda} t} \boldsymbol{S}_{11}e^{\boldsymbol{\Lambda} t} \ud t\Big)\boldsymbol{z}_{2r} \nonumber\\
&=-\boldsymbol{z}_r^T\boldsymbol{S}_{12}\boldsymbol{z}_{2r} \label{S12eqn}
\end{align}
Finally, note that $\boldsymbol{S}_{22}$ solves $\boldsymbol{\Lambda} \boldsymbol{S}_{22} + \boldsymbol{S}_{22} \boldsymbol{\Lambda} - 2\boldsymbol{S}_{12} = \boldsymbol{0}$.
 Repeated applications of integration by parts then yields the equality:
\begin{equation}
\boldsymbol{z}_{2r}^T\boldsymbol{S}_{22}\boldsymbol{z}_{2r}=\int_{0}^{\infty} t^2\Big(\sum \limits_{i=1}^{r} z_{r+1} e^{\tilde{\lambda}_i t} \Big)^2 \ud t \label{S22eqn}
\end{equation}

Combining equations (\ref{Seqn}), (\ref{S11eqn}), (\ref{S12eqn}), and (\ref{S22eqn}) gives the desired results since
$$
 \mbox{$
\boldsymbol{z}^T\boldsymbol{S}\boldsymbol{z}= \int_{0}^{\infty}\bigg[\sum\limits_{i=1}^r z_ie^{\tilde{\lambda}_i t}-t\Big(\sum \limits_{i=1}^rz_{r+i}e^{\tilde{\lambda}_i t}\Big) \bigg]^2 \ud t.
$}
$$
\end{proof}
Then it follows that the Schur complement $\boldsymbol{S}_{22}-\boldsymbol{S}_{12}\boldsymbol{S}_{11}^{-1}\boldsymbol{S}_{12}$ of $\tilde{\boldsymbol{S}}$ is also positive definite.  With the setup above, we may now commence with the proof of Theorem \ref{thm:localattractor}.

{\bf Proof of Theorem \ref{thm:localattractor}:}
It suffices to show that for any fixed point which is a local minimizer of  $\costfun(\tilde{\boldsymbol{\phi}}, \lambda(\poles))$, the eigenvalues of the Jacobian of $\lambda(\poles)$ are bounded in magnitude by 1.  As shown in \cite{krajewski}, the Jacobian of $\lambda(\poles)$ can be written as $-\boldsymbol{S}_c^{-1}\boldsymbol{K}$ 
where
$$ 
\boldsymbol{S}_c=\boldsymbol{S}_{22}-\boldsymbol{S}_{12}\boldsymbol{S}_{11}^{-1}\boldsymbol{S}_{12}\qquad \text{ and } \qquad
\boldsymbol{K}=\boldsymbol{E}\boldsymbol{R}^{-1}.
$$
First off, note that from Lemma \ref{lemma:E_pos_def}, and the fact that \Hs\ is a ZIP system by Lemma \ref{lemma:reducessslemma}, $\boldsymbol{K}$ is positive definite.  Evaluating the pencil $\boldsymbol{K}-\lambda\boldsymbol{S}_c$ at $\lambda=1$ gives
$$
\boldsymbol{\Phi}=-\boldsymbol{S}_{22}+\boldsymbol{E}\boldsymbol{R}^{-1}+\boldsymbol{S}_{12}\boldsymbol{S}_{11}^{-1}\boldsymbol{S}_{12},
$$
 This pencil is regular since $\boldsymbol{S}_c$ is positive definite by Lemma \ref{lemma:Spos}, and therefore $\det(\boldsymbol{K}-\lambda \boldsymbol{S}_c)$ is zero if and only if $\det(\boldsymbol{S}_c^{-1}\boldsymbol{K}-\lambda\boldsymbol{I})=0$. 
 
 Let $\nabla^2\costfun$ denote the Hessian of the 
cost function $\costfun(\tilde{\boldsymbol{\phi}}, \lambda(\poles))$.
 As shown in \cite{krajewski},   $\nabla^2\costfun$ can be written as
$$
\nabla^2\costfun= 
\begin{bmatrix}
\boldsymbol{I} &\boldsymbol{0}\\
\boldsymbol{0}&\boldsymbol{R}
\end{bmatrix}
\boldsymbol{M}
\begin{bmatrix}
\boldsymbol{I} &\boldsymbol{0}\\
\boldsymbol{0}&\boldsymbol{R}
\end{bmatrix},
~~~{\rm where}~~~ 
\boldsymbol{M} = 
\begin{bmatrix}
\boldsymbol{S}_{11}&\boldsymbol{S}_{12}\\
\boldsymbol{S}_{12}&\boldsymbol{S}_{22}-\boldsymbol{E}\boldsymbol{R}^{-1}
\end{bmatrix}.
$$
 
  Note that $-\boldsymbol{\Phi}$ is the Schur complement of $\boldsymbol{M}$. Hence, if the fixed point is a local minimimum, then $-\boldsymbol{\Phi}$ must be positive definite and so for $\lambda=1$ the pencil is negative definite. Since both $\boldsymbol{K}$ and $\boldsymbol{S}_c$ are positive definite,
 there exists a nonsingular transformation $\boldsymbol{Z}$ by which the quadratic form
$\displaystyle 
\boldsymbol{y}^T(\boldsymbol{K}-\lambda\boldsymbol{S}_c)\boldsymbol{y}
$
is transformed into
$\displaystyle
\boldsymbol{z}^T(\boldsymbol{\Lambda}-\lambda\boldsymbol{I})\boldsymbol{z},
$
where $\boldsymbol{\Lambda}$ is a diagonal matrix formed from the solutions of 
\begin{equation}\label{solutions}
\det(\boldsymbol{K}-\lambda\boldsymbol{S}_c)=0.
\end{equation}
Thus, the solutions of (\ref{solutions}) correspond to the eigenvalues of $\boldsymbol{S}_c^{-1}\boldsymbol{K}$. 
$\boldsymbol{\Lambda}-\boldsymbol{I}$ must be negative definite since $\boldsymbol{\Phi}$ is, and therefore the eigenvalues of the $\boldsymbol{S}_c^{-1}\boldsymbol{K}$ must be real-valued and less than one.  Furthermore, note that $\boldP=\boldsymbol{S}_c^{-1}\boldsymbol{K}$ solves the Lyapunov equation
$$
\boldP\boldsymbol{S}^{-1}_c+\boldsymbol{S}^{-1}_c\boldP^T=2\boldsymbol{S}^{-1}_c\boldsymbol{K}\boldsymbol{S}^{-1}_c,
$$
so by the standard inertia result, all the eigenvalues of $\boldsymbol{S}_c^{-1}\boldsymbol{K}$ are positive, and the desired result follows.   $\Box$

\bibliographystyle{plainnat}
\bibliography{references}

\end{document}